\documentclass[11pt]{article}%
\usepackage{algorithm}
\usepackage{algpseudocode}
\usepackage{amssymb}
\usepackage{amsfonts}
\usepackage{amsmath}
\usepackage{graphicx}%
\usepackage[caption=false]{subfig}
\usepackage{tikz}
\usetikzlibrary{shapes,arrows,positioning,calc}
\usepackage{pgfplots}
\usepackage{verbatim}
\usepgfplotslibrary{statistics}
\pgfplotsset{compat=1.8}
\usepackage{amsmath,amssymb,amsthm}
\usepackage[numbered]{matlab-prettifier}

\usepackage{url}
\usepackage{graphicx}

\newtheorem{theorem}{Theorem}

\newtheorem{lemma}[theorem]{Lemma}

\newtheorem{remark}[theorem]{Remark}

\newcommand{\df}[2]{\frac{\textnormal{d}  #1}{\textnormal{d} #2}}

\newcommand{\pd}[2]{\dfrac{\partial #1}{\partial #2}}

\renewcommand{\O}{\mathcal{O}}

\newcommand{\eps}{\varepsilon}

\newcommand{\fbr}{f_{\textnormal{br}}} 
\newcommand{\fcr}{f_{\textnormal{cr}}} 
\newcommand{\yinfty}{y_{\E}} 
\newcommand{\zinfty}{z_{\E}} 

\newcommand{\E}{{\mathchoice{}{}{\scriptscriptstyle}{} \textnormal{E}}}

\begin{document}

\title{A minimal model for adaptive SIS epidemics}

\author{Massimo A. Achterberg$^{1,}$\thanks{M.A.Achterberg@tudelft.nl} \ and Mattia Sensi$^{2,}$\thanks{mattia.sensi@inria.fr}}
\date{\small $^1$Faculty of Electrical Engineering, Mathematics and Computer Science \\ Delft University of Technology, P.O. Box 5031, 2600 GA Delft, The Netherlands \\
$^2$MathNeuro Team, Inria at Universit\'e C\^ote d'Azur, 2004 Rte des Lucioles, 06410 Biot, France}
\maketitle




\begin{abstract}
    The interplay between disease spreading and personal risk perception is of key importance for modelling the spread of infectious diseases. We propose a planar system of ordinary differential equations (ODEs) to describe the co-evolution of a spreading phenomenon and the average link density in the personal contact network. Contrary to standard epidemic models, we assume that the contact network changes based on the current prevalence of the disease in the population, i.e.\ the network adapts to the current state of the epidemic. We assume that personal risk perception is described using two functional responses: one for link-breaking and one for link-creation. The focus is on applying the model to epidemics, but we also highlight other possible fields of application. We derive an explicit form for the basic reproduction number and guarantee the existence of at least one endemic equilibrium, for all possible functional responses. Moreover, we show that for all functional responses, limit cycles do not exist.
\end{abstract}

\section{Introduction}

Classical compartmental models in epidemiology rely on the widely accepted assumption of homogeneous mixing. While this assumption greatly simplifies the analysis of such models, its interpretation clashes with the reality of human interaction. Network models have been proposed and studied to include a more realistic pattern of connections between individuals \cite{ReviewPaperSIS}.

Most network-based research focuses on contact patterns that remain fixed over time. However, real-world contacts vary over time, especially during epidemic outbreaks, because of individual decisions of people to avoid contact with other people. Such networks are called \emph{adaptive} networks, because the network adapts itself to the spread of the disease \cite{GrossAdaptiveReview}. 

The excellent review by Verelst \emph{et al}.\ \cite{verelst2016reviewBehaviorEpidemics} provides an overview of various practical approaches for the mathematical modelling of the interplay between disease and human behavior. A multi-layer approach was adopted by Sahneh \emph{et al}.\ \cite{sahneh_2019_awareness_SIS}, where one layer describes the disease transmission and another layer the awareness of individuals about the disease. Gross \emph{et al}.\ \cite{GrossEtAl2006} proposed a rewiring mechanism, which rewires the link between two connected susceptible-infected nodes to two susceptible nodes. Kiss \emph{et al}.\ \cite{RandomLinkDeActivate} (and independently Achterberg \emph{et al}.\ \cite{G-ASIS}) introduced a Link Activation-Deactivation model, in which links can be broken or created between two specified types of nodes. Jolad \emph{et al}.\ \cite{jolad2012adaptiveresponse} assumes that all individuals have a preferred number of neighbours, subject to random link addition and removals. The preferred degree is taken to be a function of the current number of infected nodes in the network. Brauer \cite{brauer2011simple} discusses an SIR model in which a certain percentage of the links is removed. The removal percentage is larger if the link is connected to infected nodes rather than susceptible nodes. All such models capture a particular aspect of human behavior on disease dynamics, but most are so complicated, that an exact analysis is completely infeasible.

In this work, we propose a minimal model consisting of two ODEs, one for the viral prevalence in the population using the NIMFA equations \cite{IntroNIMFA}, and one for the weights of the links in the contact network. We model the creation and removal of edges as an overall increase or decrease of the weight on the edges. We call the model adaptive NIMFA (aNIMFA), in line with earlier work \cite{AchterbergMieghemAdaptiveMeanField}. The core aspect of aNIMFA are the \emph{functional responses} of individuals to create or break links in the network, based on the current number of infected people. In predator-prey systems like Volterra-Lotka dynamics, Holling introduced functional responses to describe the food intake by predators as a function of the number of available prey \cite{holling1959volterralotka}. A preliminary analysis of the aNIMFA model was performed by Achterberg and Van Mieghem \cite{AchterbergMieghemAdaptiveMeanField}, but only for specific functional responses. We extend the results from \cite{AchterbergMieghemAdaptiveMeanField} by considering general functional responses and by providing a more detailed analysis.

The aNIMFA model is not limited to modelling epidemic spread, but can be utilised for describing general spreading phenomena, including opinion dynamics, Maki-Thompson rumour spread, and others. In the context of epidemics, one would expect the removal (resp. creation) of links to be directly (resp. inversely) proportional to the prevalence. For other spreading phenomena, such as rumor spreading, this might not be the case, and other choices for the functional responses can be made. The simplicity of the aNIMFA model makes it a promising tool for future generalizations and for the integration of more complex mechanisms. 

Lastly, we consider the situation where the network changes slowly compared to the spread of the disease, and we study the qualitative behaviour of the resulting model using Geometric Singular Perturbation Theory (GSPT) \cite{fenichel1979geometric,kuehn2015multiple}. Techniques from GSPT have been applied to epidemiological models in which the loss of immunity and demography are slow compared to infection and recovery from a disease in \cite{Mattia2021FastSlow,jardon2021geometric}. Additionally, similar techniques were applied to epidemics modelling e.g. in \cite{brauer2019singular,bravo2021discrete,schecter2021geometric,zhang2009singular,aguiar2021time}. 

The paper is structured as follows. We introduce the aNIMFA model in Sec.\ \ref{sec_model} and provide a thorough analysis in Sec.\ \ref{sec_model_analysis}. Then we consider several examples of functional responses in Sec.\ \ref{sec_examples}. We study a slowly evolving network in Sec.\ \ref{sec_slow_network} using Geometric Singular Perturbation Theory and present a conclusion in Sec.\ \ref{sec_conclusion}.

\section{The aNIMFA model}\label{sec_model}
Consider a well-mixed population of $N$ individuals, subject to the spread of a disease. The mean-field dynamics of the SIS process for a well-mixed population is generally described in terms of the average fraction of infected nodes $y(t)$, also known as the \emph{prevalence}. The governing equation equals
\begin{equation}\label{eq_diff_y_temp}
	\df{y}{t} = -\delta y + \beta y (1-y) z,
\end{equation}
where the curing process is denoted by its rate $\delta$, the infection process by the corresponding rate $\beta$ and $z$ is the link density. In the first term in Eq.\ (\ref{eq_diff_y_temp}), the prevalence decreases proportional to the current number of infected cases. The second term in Eq.\ (\ref{eq_diff_y_temp}) increases the prevalence because of contact between infected $y$ and susceptible $1-y$ nodes. Because of the homogeneous mixing, we multiply with the link density $z$ to obtain the average number of contacts. Equation (\ref{eq_diff_y_temp}) directly follows from the N-Intertwined Mean-Field Approximation (NIMFA) equations \cite{IntroNIMFA} when considering homogeneous infection and curing rates, symmetric initial conditions and a complete graph with weight $z$.

Contrary to the static SIS process, we assume that the link weight $z(t)$ is varying over time and its dynamics is governed by a link-breaking and a link-creation process. Then the link density $z(t)$ changes over time as
\begin{equation}\label{eq_diff_z_temp}
	\df{z}{t} = - \zeta z \fbr(y) + \xi (1-z) \fcr(y),
\end{equation}
where $\zeta$ is the link-breaking rate, $\xi$ the link-creation rate and $\fbr(y)$ and $\fcr(y)$ are the functional responses to the link-breaking and link-creation process, respectively. We assume the parameters $\delta$, $\beta$, $\zeta$, $\xi$ to be positive. The link weight $z$ has been normalised, such that $z=1$ is the maximum link weight (corresponding to a complete graph) and $z=0$ corresponds to an empty graph (no connections, so the link weight is zero).

Equations (\ref{eq_diff_y_temp}) and (\ref{eq_diff_z_temp}) can be simplified by introducing the scaled time $\tilde{t} = \delta t$. We additionally introduce the \emph{effective infection rate} $\tau=\beta/\delta$ and the \emph{effective link-breaking rate} $\omega=\zeta/\xi$. Using the transformations $\tilde{\zeta}=\zeta/\delta$ and $\tilde{\xi}=\xi/\delta$, the well-mixed adaptive NIMFA (aNIMFA) equations are obtained (after dropping the tildes, for ease of notation)
\begin{subequations}\label{eq_aNIMFA}
    \begin{align}
    	\df{y}{t} &= -y + \tau y (1-y) z, \label{eq_diff_y} \\
    	\df{z}{t} &= - \zeta z \fbr(y) + \xi (1-z) \fcr(y), \label{eq_diff_z} \\
    	& \qquad \qquad \text{feasible region } 0 \leq y \leq 1, 0 \leq z \leq 1. \nonumber
\end{align}
\end{subequations}
The initial conditions $y(0)\in[0,1]$ and $z(0)\in[0,1]$ describe the initial prevalence and link-density, respectively. We assume that the functional responses $\fbr(y)$ and $ \fcr(y)$ are non-negative, sufficiently regular functions on the interval $0 \leq y \leq 1$. We exclude the possibility that $\fbr(y)=0$ and $\fcr(y)=0$ for all $y$, as in this case, one reduces to a static $z$-regular graph.

\section{Analysis of the model}\label{sec_model_analysis}
Prior to confining ourselves to specific link-breaking and link-creation functions $\fbr$ and $\fcr$, we first derive general results for the aNIMFA model.

\begin{lemma}\label{lem:bounded}
Consider a solution of system \eqref{eq_aNIMFA} starting at $(y(0),z(0)) \in [0,1]^2$. Recall that $\fbr(y),\fcr(y)\geq 0$ for all $y\in [0,1]$. Then, $(y(t),z(t)) \in [0,1]^2$ for all $t \geq 0$.
\end{lemma}
\begin{proof}
    We calculate
    $$
    \df{y}{t}\bigg|_{y=0}=0, \quad \df{y}{t}\bigg|_{y=1}=-1<0, \quad \df{z}{t}\bigg|_{z=0}=\xi \fcr(y) \geq 0, \quad \df{z}{t}\bigg|_{z=1}=-\zeta \fbr(y) \leq 0,
    $$
    which proves the forward invariance of $[0,1]^2$.
\end{proof}

\subsection{Disease-free equilibrium}
The aNIMFA model always has one steady state $y_0=0$, which corresponds to the situation in which no infected individuals are present in the population. In line with the literature, we call this steady state the \emph{disease-free equilibrium} (DFE). The DFE of the mean-field equations (\ref{eq_aNIMFA}) equals
\begin{align*}
	y_0 &= 0, \\
	z_0 &= \begin{cases}
		\frac{\fcr(0)}{\omega \fbr(0)+\fcr(0)} & \qquad \text{if } \fbr(0) \neq 0 \text{ or } \fcr(0) \neq 0, \\
		\text{free var} & \qquad \text{if } \fbr(0)=\fcr(0)=0.
	\end{cases}
\end{align*}

\subsection{Endemic equilibria}
Depending on the choice of the functional responses $\fbr$ and $\fcr$, multiple additional steady states may exist, which are called the \emph{endemic equilibria} (EE). The endemic equilibria are the solutions of the non-linear equation
\begin{equation}\label{eq_EE_implicit}
	\omega \fbr(\yinfty) = (\tau-1) \fcr(\yinfty) - \tau \yinfty \fcr(\yinfty),
\end{equation}
and the corresponding steady-state link density $\zinfty$ follows as
\begin{equation}\label{eq_EE_z}
	\zinfty = \frac{1}{\tau(1-\yinfty)}.
\end{equation}
We remark that the solution $\yinfty=1$ is never a valid EE for any functional responses $\fbr$ and $\fcr$, which follows immediately from substituting $\yinfty=1$ into Eq.\ (\ref{eq_EE_implicit}). Hence, all EE are in the open interval $(0,1)$; in Theorem \ref{thm_EE_exist} we prove that $\zinfty\in (0,1]$ as well.

Endemic equilibria may not need to exist for all parameter values. In particular, the point where the disease-free equilibrium loses stability often coincides with the birth of an endemic equilibrium. This happens when the \emph{basic reproduction number} $R_0$, which is the number of secondary infections produced by one average infected individual in an otherwise susceptible population, crosses the threshold value of $1$. Fortunately, there is always a region in the $(\tau,\omega)$-space where an EE exists, as we prove in Theorem~\ref{thm_EE_exist}.

\begin{theorem}[Existence of EE]\label{thm_EE_exist}
	The non-linear algebraic equation (\ref{eq_EE_implicit}) always has at least one solution for some $(\tau,\omega)$-region. In other words, there always exists at least one endemic state.
\end{theorem}
\begin{proof}
	We prove the theorem by showing that the reverse cannot hold, i.e. we look for functions $\fbr$ and $\fcr$ for which no solution exists for all $(\tau,\omega)$-values. Now we introduce the function $h(\yinfty)$ defined as
	\begin{equation*}
		h(\yinfty) = \omega \fbr(\yinfty) + \fcr(\yinfty)- \tau(1-\yinfty) \fcr(\yinfty).
	\end{equation*}
	We remark that solutions of $h(\yinfty)=0$ correspond to solutions of (\ref{eq_EE_implicit}). The function $h$ is sufficiently regular, because $h$ is the composite of such functions $\fbr$ and $\fcr$. According to the Intermediate Value Theorem, if there exists $\yinfty$ and $\yinfty'$ such that $h(\yinfty)>0$ and $h(\yinfty')<0$, then there must exists some $\yinfty''$ for which holds that $h(\yinfty'')=0$. To guarantee that solutions do not exist, we must prove that either $h(\yinfty) > 0$ or $h(\yinfty)<0$ for all $\yinfty$. Let us consider the first case; the second case can be proven with an identical strategy. The function $\fcr$ is non-negative and non-identically 0, thus there must exist some $\yinfty'\in (0,1)$ such that $\fcr(\yinfty')>0$. To ensure positivity of the function $h$ for $\yinfty'$, we find the condition
	\begin{equation}\label{eq_condition1}
		\omega \fbr(\yinfty') + \fcr(\yinfty') > \tau(1-\yinfty') \fcr(\yinfty').
	\end{equation}
	Suppose $\fbr(\yinfty')=0$, then the equation simplifies to
	\begin{equation*}
		1 > \tau (1-\yinfty'),
	\end{equation*}
	which is not satisfied unconditionally, that is, for all values of $\tau$, except if we would allow $\yinfty'=1$ as a solution (which is, fortunately, excluded as a steady-state solution, see Eq.\ (\ref{eq_diff_y})). 
 
 If $\fbr(\yinfty')>0$, condition (\ref{eq_condition1}) can also not be satisfied unconditionally, because $\tau$ and $\omega$ appear on opposite sides of the equation and neither side is zero, so condition (\ref{eq_condition1}) cannot be true for all $\tau$ and $\omega$. We conclude that there is always a non-empty $(\tau,\omega)$-region where a solution exists. 
	
    Finally, $h$ can be rewritten as
    $$h(\yinfty) = \omega \fbr(\yinfty) + (1- \tau(1-\yinfty)) \fcr(\yinfty).$$
    The assumption $\omega >0$, combined with $\fbr(\yinfty), \fcr(\yinfty)\geq 0$ and the fact that $h(\yinfty)=0$ leads to the condition that $\tau (1-\yinfty)\geq 1$. This ensures $\zinfty \in (0,1]$ (recall Eq.\ (\ref{eq_EE_z})).
\end{proof}

\subsection{Basic reproduction number}\label{basrepnum}
In this section, we provide an expression for the \emph{basic reproduction number} $R_0$, also known as the epidemic threshold, using the next generation matrix method, which was first introduced in \cite{diekmann1990definition}, then generalized in \cite{VandenDriesscheWatmough} (see also \cite{diekmann2010construction}). Even though the compartmental component of system (\ref{eq_aNIMFA}) is one-dimensional (the equation for the link density does not count) and the analysis could also be done by local stability analysis, we have chosen for the next generation matrix method due to its widely spread use.

We rewrite the first equation of \eqref{eq:Jac0} as $J_{11}=M_{11}-V_{11}$, with $M_{11},V_{11} > 0$. The only  such splitting possible, assuming $\fcr(0)>0$, is
\begin{equation*}
	M_{11}=	\tau \zinfty, \quad 	V_{11}= 1.
\end{equation*}
Then, the basic reproduction number $R_0$ is $M_{11}V_{11}^{-1}$, i.e.
\begin{equation}\label{eq_basic_rep}
   R_0 = \tau \frac{\fcr(0)}{\omega \fbr(0)+\fcr(0)}.
\end{equation}
For the case $\fcr(0)=0$, the method does not apply: this models a particularly degenerate situation in our model, as absence of the disease does not increase the connectivity strength. We remark that, in such a simple one-dimensional context, the next generation matrix method naturally coincides with the linear stability analysis of the DFE which we carry out in the next section.

\subsection{Linear stability analysis}
We analyse the linear stability of the steady states by computing the Jacobian of Eq. (\ref{eq_aNIMFA}) as
\begin{equation*}
	J = \begin{pmatrix}
		-1+\tau(1-2\yinfty)\zinfty & \tau \yinfty(1-\yinfty) \\
		-\zeta \zinfty \fbr'(\yinfty) + \xi (1-\zinfty) \fcr'(\yinfty) & - \zeta \fbr(\yinfty) - \xi \fcr(\yinfty)
	\end{pmatrix}.
\end{equation*}
For the disease-free equilibrium $\yinfty=0,\zinfty=z_0$, we find
\begin{equation}\label{eq:Jac0}
	J(0,z_0) = \begin{pmatrix}
		-1+\tau z_0 & 0 \\
		-\zeta z_0 \fbr'(0) + \xi (1-z_0) \fcr'(0) & - \zeta \fbr(0) - \xi \fcr(0)
	\end{pmatrix}.
\end{equation}
Since the Jacobian for the disease-free equilibrium is lower-triangular, the eigenvalues are $\lambda_1 = -1+\tau z_0$ and $\lambda_2 = -\zeta \fbr(0)-\xi \fcr(0)$. The eigenvalues are always real, so (un)stable spirals cannot be observed. We now consider several cases.
\begin{enumerate}
	\item \textbf{Case $\fbr(0)=0$ and $\fcr(0)=0$} \\
	The eigenvalues are $\lambda_1= -1+\tau z_0$ and $\lambda_2 = 0$, which makes the stability undeterminable using linear stability analysis.
	\item \textbf{Case $\fbr(0)=0$ and $\fcr(0)>0$} \\
	The eigenvalues are $\lambda_1 = -1 + \tau$ and $\lambda_2 = -\xi \fcr(0)$. Thus the disease-free equilibrium is a stable node if $\tau < 1$ and an unstable node if $\tau > 1$. For $\tau=1$, the stability is undetermined. In this case, $z_0=1$.
	\item \textbf{Case $\fbr(0)>0$ and $\fcr(0)=0$} \\
	The eigenvalues are $\lambda_1 = -1$ and $\lambda_2 = -\zeta \fbr(0)$, thus the DFE is a stable node. In this case, $z_0=0$.
	\item \textbf{Case $\fbr(0)>0$ and $\fcr(0)>0$} \\
	The eigenvalues are $\lambda_1 = -1 + \tau \frac{\fcr(0)}{\omega \fbr(0) + \fcr(0)}$ and $\lambda_2 = -\zeta \fbr(0)-\xi \fcr(0)$. Eigenvalue $\lambda_2 < 0$, thus the stability solely depends on $\lambda_1$. The disease-free equilibrium is a stable node if $\tau < \frac{\omega \fbr(0)+\fcr(0)}{\fcr(0)}$, an unstable node if $\tau > \frac{\omega \fbr(0)+\fcr(0)}{\fcr(0)}$ and is undetermined otherwise. 
\end{enumerate}

We remark that, in cases 2 and 4, the linear stability or instability of the DFE coincides with $R_0$ derived in Section \ref{basrepnum} being smaller or bigger than 1.

Unfortunately, we cannot directly analyse the stability of the endemic equilibria, because (i) we do not know $\yinfty$ nor $\zinfty$ and (ii) we require the functions $\fbr$ and $\fcr$ and its derivatives $\fbr'$ and $\fcr'$ to determine the stability. Moreover, the existence of multiple endemic equilibria rules out the possibility of finding a Lyapunov function to prove the global stability of system (\ref{eq_aNIMFA}). Nevertheless, for specific functional responses $\fbr$ and $\fcr$ that have only a single EE, one could attempt to construct a Lyapunov function, which is outside the scope of this paper.

\subsection{Global stability}
Before proving global stability, we first consider limit cycles of the aNIMFA model, for which we invoke the Bendixson-Dulac theorem.
\begin{theorem}[Bendixson-Dulac]\label{thm:bend} 
    If there exists a $C^1$-function $\phi (y,z)$ such that the expression
    \begin{equation}\label{eq_Bendixson-Dulac}
        F(y,z) = \pd{(\phi f)}{y} + \pd{(\phi g)}{z}
    \end{equation}
    has the same sign $(\neq 0)$ almost everywhere in a simply connected region $R$, then the planar autonomous system
    \begin{align*}
        \df{y}{t} = f(y,z), \\
        \df{z}{t} = g(y,z),
    \end{align*}
    has no non-constant periodic solutions lying entirely within the region $R$.
\end{theorem}
A proof of Theorem \ref{thm:bend} can be found in \cite{bendixson1901courbes}, or in \cite{li1993bendixson} for the $n$-dimensional case. We now apply Theorem \ref{thm:bend} to prove that system \eqref{eq_aNIMFA} admits no periodic solutions.
\begin{theorem}
    System \eqref{eq_aNIMFA} admits no non-trivial periodic solutions.
\end{theorem}
\begin{proof}
We verify the Bendixson-Dulac criterion using $\phi(y,z) = \frac{1}{yz}$ for our system (\ref{eq_aNIMFA}) in the region $R=(0,1)^2$. We find
\begin{align*}
    \phi f &= -\frac{1}{z} + \tau (1-y), \\
    \phi g &= - \zeta \frac{\fbr(y)}{y} + \xi \frac{\fcr(y)}{yz} - \xi \frac{\fcr(y)}{y}.
\end{align*}
Filling in Eq.\ (\ref{eq_Bendixson-Dulac}) gives
\begin{equation*}
    F(y,z) = -\tau - \xi \frac{\fcr(y)}{y z^2}.
\end{equation*}
Since $y,z > 0$ and $\fcr(y)$ is a non-negative function, we conclude that $F < 0$ in the whole region $R=(0,1)^2$ and there cannot exist any limit cycles.
\end{proof}

Recall that the DFE is locally (hence, globally) unstable when $R_0>1$. We make the following remark:

\begin{remark}\label{rem:homocl}
Assume that $R_0>1$, and that the DFE is on the repelling part of the $z$-axis $\{ y=0, z>\frac{1}{\tau} \}$. Then, we can exclude the possibility of homoclinic orbits to the DFE, whose stable manifold is the $z$-axis. Under these assumptions, as a consequence of the Poincar{\'e}-Bendixson theorem, we conclude that in system (\ref{eq_aNIMFA}) the endemic equilibrium, if it is unique, is globally asymptotically stable. If multiple endemic equilibria exist, or the DFE is in the attracting part of the $z$-axis, i.e.\ $\{ y=0,z < \frac{1}{\tau}\}$, no general conclusions can be drawn.
\end{remark}

\section{Examples}\label{sec_examples}
In the previous section, we derived several general results for the aNIMFA model. However, certain properties, like the number and stability of the endemic states, could not be determined for general functional responses. In this section, we investigate several examples of functional responses $\fbr$ and $\fcr$, whereby we primarily focus on epidemiological applications. Then, by assumption, the link-breaking rule $\fbr(y)$ is likely to be increasing with the prevalence $y$ and the link-creation rule $\fcr(y)$ is exactly opposite. The aNIMFA model is, however, more versatile and can be applied to other spreading phenomena, including opinion dynamics, cascading failures and information transport in the human brain. These spreading phenomena are often more complex than SIS epidemic spread, thus requiring more complex (maybe even non-monotone) functional responses $\fbr$ and $\fcr$.

\subsection{Example 1: Random Link-Activation Deactivation \cite{RandomLinkDeActivate}}
Presumably the easiest functional responses are those that are totally unaffected by the current number of infected cases. Then the network density evolves independently of the epidemic prevalence. This model is known as the Random Link-Activation Deactivation (RLAD) model~\cite{RandomLinkDeActivate}. In this model, each link in the underlying network can be randomly created or broken, with rates $\xi$ and $\zeta$ respectively. Mathematically, we require that the functional responses $\fbr$ and $\fcr$ are constant and for simplicity, we consider $\fbr(y)=\fcr(y)=1$, and hence system (\ref{eq_aNIMFA}) becomes
\begin{subequations}\label{random_system}
	\begin{align}
		\df{y}{t} &= - y + \tau y (1-y) z, \\
		\df{z}{t} &= - \zeta z + \xi (1-z).
	\end{align}
\end{subequations}
Then, the basic reproduction number as defined in Eq.\ \eqref{eq_basic_rep} is $R_0=\frac{\tau}{1+\omega}$. In this simple example, the governing equation (\ref{eq_diff_z}) for the link-density $z(t)$ is decoupled from the prevalence $y(t)$ and can be solved directly;
\begin{equation*}
	z(t) = \frac{1}{1+\omega} + \left( z_0 - \frac{1}{1+\omega}\right) e^{-(\xi+\zeta) t},
\end{equation*}
where the effective link-breaking rate $\omega = \zeta/\xi$. If the exponential decays sufficiently fast (i.e. $\xi+\zeta$ is large), the network density quickly converges to $z=1/(1+\omega)$. Substituting $z=1/(1+\omega)$ into Eq.~(\ref{eq_diff_y}) and solving yields the famous logistic equation \cite{VerhulstLogistic} for the prevalence;
\begin{equation}\label{eq_RLAD_y}
    y(t) = \frac{\yinfty}{1+e^{-K(t-t_0)}},
\end{equation}
where $\yinfty = 1 - \frac{1+\omega}{\tau}$ is the steady-state prevalence, $K = \tau-1$ is the growth rate and $t_0 = \frac{1}{K} \ln\left(\frac{\yinfty}{y_0}-1\right)$ is the \emph{inflection point}, better known as the \emph{epidemic peak}.

The time-varying prevalence $y(t)$, given by Eq.~(\ref{eq_RLAD_y}), converges to a unique, non-zero, steady-state prevalence $\yinfty>0$ if $\tau>1+\omega$. Otherwise, for $\tau < 1+\omega$, the prevalence decreases exponentially to zero. The same result follows from linear stability analysis. The DFE, given by $(\yinfty, \zinfty)=(0,\frac{1}{1+\omega})$, is asymptotically stable for $\tau < 1+\omega$, unstable for $\tau > 1+\omega$ and undetermined for $\tau=1+\omega$. The unique endemic equilibrium is given by $(\yinfty, \zinfty)=\left( 1-\frac{1+\omega}{\tau}, \frac{1}{1+\omega}\right)$, which is in the biologically feasible region only if $R_0>1$, and coincides with the DFE when $R_0=1$. The Jacobian is
\begin{equation*}
	J = \begin{pmatrix}
		1-\frac{\tau}{1+\omega} & (1+\omega) \left(1-\frac{1+\omega}{\tau}\right) \\
		0 & - \zeta - \xi
	\end{pmatrix}.
\end{equation*}
The eigenvalues are $\lambda_1 = 1-\frac{\tau}{1+\omega}$ and $\lambda_2 = - \zeta - \xi < 0$. Thus the EE is a stable node if $R_0>1$, unstable node if $R_0<1$ and undetermined for $\tau=1+\omega$. As we remarked above, the case $R_0<1$ leads to $\yinfty<0$ which is biologically infeasible. The steady states and their behavior of the RLAD model is shown in the table below.
\begin{table}[H]
	\centering
\begin{tabular}{ccc}
	\textbf{Example 1:} $\fbr(y)=\fcr(y)=1$ & if $R_0 \leq 1$ & if $R_0 > 1$ \\
	\hline
	Disease-free state $\left(0,\frac{1}{1+\omega}\right)$ & stable node & unstable node \\
	Endemic equilibrium $\left( 1-\frac{1+\omega}{\tau}, \frac{1}{1+\omega}\right)$ & unstable node & stable node
\end{tabular}
\end{table}
Since the link-dynamics is decoupled from the disease dynamics in the RLAD model, the behaviour of the RLAD model is very similar to the static SIS model and undergoes the usual transcritical bifurcation, except that the basic reproduction number $R_0$ is a function of the effective link-breaking rate $\omega$. For other functional responses $\fbr$ and $\fcr$, we expect different behaviour, which will be investigated in the upcoming examples.

\subsection{Example 2: Epidemics: $\fbr(y)=y$, $\fcr(y)=1$}
Contrary to the randomly evolving links in Example 1, we expect that genuine epidemic outbreaks affect the number of contacts of people. We consider the simple case where the link-breaking process $\fbr(y)=y$ is a linear function of the prevalence, but the link-creation process remains independent from the total number of infections [$\fcr(y)\equiv 1$]. Then, the governing equations become
\begin{subequations}\label{eq_system}
	\begin{align}
		\df{y}{t} &= - y + \tau y (1-y) z, \\
		\df{z}{t} &= - \zeta z y + \xi (1-z).
	\end{align}
\end{subequations}
The basic reproduction number as defined in Eq.\ \eqref{eq_basic_rep} is $R_0=\tau$. The disease-free equilibrium $(\yinfty, \zinfty)=(0, 1)$ is a stable node if $\tau <1$, an unstable node if $\tau > 1$ and is otherwise undetermined. The unique EE follows from Eq.\ (\ref{eq_EE_implicit}) as $(\yinfty, \zinfty)=\left(\frac{\tau-1}{\tau+\omega},\frac{\tau+\omega}{\tau(1+\omega)}\right)$ and exists in the biologically feasible region for $\tau>1$.

We now show that the unique EE is locally stable for this specific choice of $\fbr$ and $\fcr$. The Jacobian around the EE equals
\begin{equation*}
	J = \begin{pmatrix}
		-1+\tau (1-2\yinfty) \zinfty & \tau \yinfty (1-\yinfty) \\
		-\zeta \zinfty & -\zeta \yinfty - \xi
	\end{pmatrix}=\begin{pmatrix}
		-\frac{\tau-1}{1+\omega} & \frac{\tau (\tau-1)(\omega+1)}{(\tau+\omega)^2} \\
		-\zeta \frac{\tau+\omega}{\tau(1+\omega)} & -\zeta \frac{\tau(\omega+1)}{\omega(\tau+\omega)}
	\end{pmatrix}.
\end{equation*}
Clearly, for $\tau>1$ and $\zeta,\omega>0$, we have $J_{1,1},J_{2,1},J_{2,2}<0$ and $J_{1,2}>0$; hence, $\text{tr}(J)<0$ and $\text{det}(J)>0$, which implies that the real parts of its eigenvalues are negative. Hence, the EE is locally stable. Following Remark \ref{rem:homocl}, the EE is also globally asymptotically stable for $R_0>1$, which is a consequence of the absence of limit cycles guaranteed by Bendixson-Dulac and the fact that the DFE is unstable for $R_0>1$.

We summarize the stability of the two equilibria in the following table and present simulations of the two possible behaviours of system \eqref{eq_system} in Figure \ref{fig_ex2}.

\begin{table}[H]
	\centering
	\begin{tabular}{ccc}
		\textbf{Example 2:} $\fbr(y)=y,\fcr(y)=1$ & if $R_0 \leq 1$ & if $R_0> 1$ \\
		\hline
		Disease-free state $\left(0,1\right)$ & stable node & unstable node \\
		Endemic equilibrium $\left(\frac{1-\frac{1}{\tau}}{1+\frac{\omega}{\tau}},\frac{1+\frac{\omega}{\tau}}{1+\omega}\right)$ & unstable spiral & stable spiral 
	\end{tabular}
\end{table}

\begin{figure}[H]
	\centering
	\subfloat[\label{fig_ex2_below} $R_0=0.8$]{%
		\includegraphics[width=0.49\textwidth]{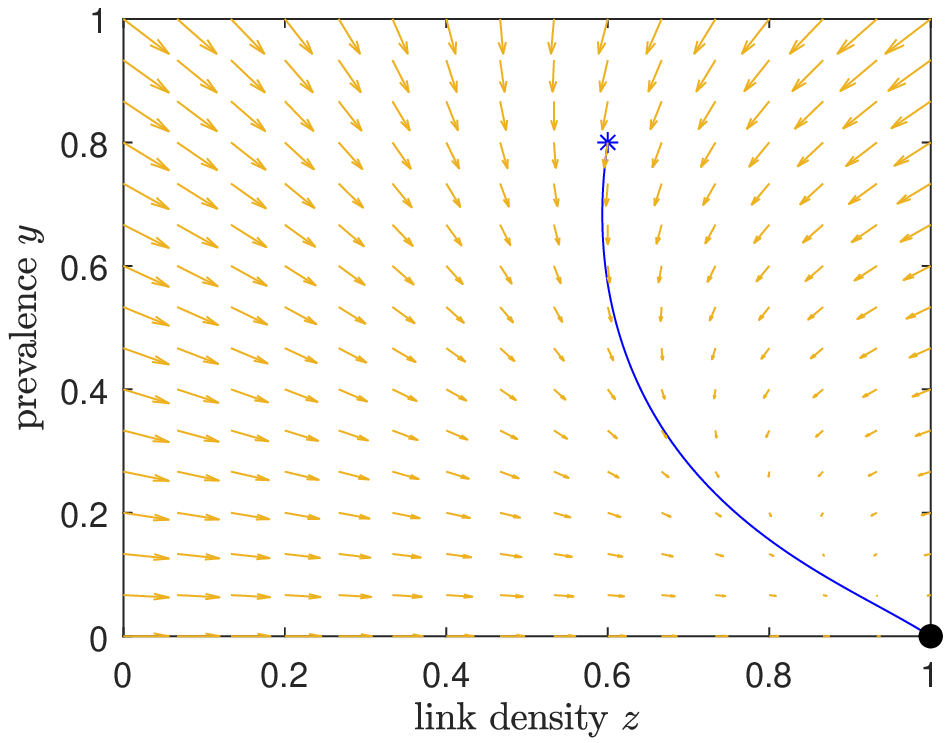}
	}
	\subfloat[\label{fig_ex2_above} $R_0=5.4$]{%
		\includegraphics[width=0.49\textwidth]{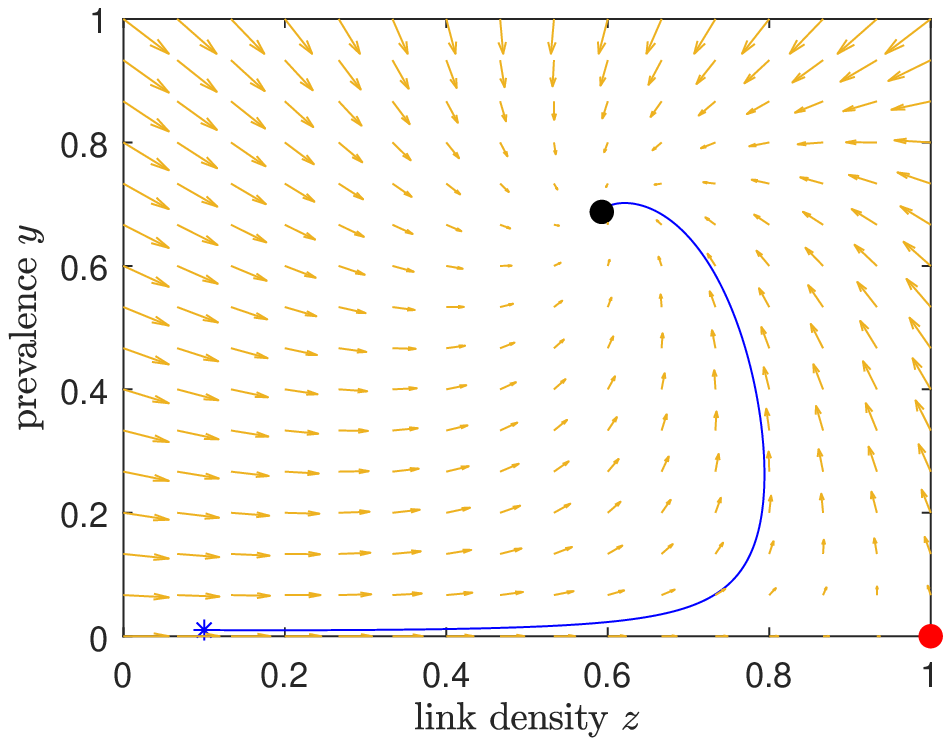}
	}
	\caption{Dynamics for Example 2. Starting point: asterisk; stable equilibrium: black dot; unstable equilibrium: red dot. (a) If $R_0<1$, any initial condition converges to the DFE. (b) If $R_0>1$, the unique EE is globally stable. The other parameters are, for simplicity, $\zeta=\xi=1$.}
	\label{fig_ex2}
\end{figure}

Comparing this example to Example 1, the behaviour is different in two ways. First, the basic reproduction number $R_0 = \tau$  does not depend on the link-breaking rate~$\zeta$ and link-creation rate~$\xi$. Second, the endemic equilibrium remains a globally stable equilibrium, but in this case, the endemic equilibrium shows spiral behaviour around the equilibrium.

\subsection{Example 3: The Adaptive SIS model}
The adaptive SIS (ASIS) model was introduced by Guo \emph{et al}.\ \cite{Guo2013} to describe the responses of individuals to an on-going pandemic. In particular, it was assumed that links are broken between susceptible and infected nodes and (re)created between susceptible nodes. The aNIMFA approximation of the ASIS model was already analysed in \cite{AchterbergMieghemAdaptiveMeanField} and the functional responses were derived as $\fbr(y)=2 y (1-y)$ and $\fcr(y)=(1-y)^2$. The link-breaking response is similar to Example 2, but the term $1-y$ was added to account for the fact that for large epidemic outbreaks, the susceptible population may be depleted and the possibility to break links between susceptible and infected individuals decreases, simply because of the lack of susceptible individuals. The factor 2 is a conversion factor from the original Markovian model; we keep this factor for consistency with \cite{AchterbergMieghemAdaptiveMeanField}. The link-creation response is more intuitive; we expect many links to be created if the disease is almost nonexistent. Hence, we are considering the system of ODEs
\begin{subequations}\label{ASIS_system}
	\begin{align}
		\df{y}{t} &= - y + \tau y (1-y) z, \\
		\df{z}{t} &= - 2\zeta zy(1-y) + \xi (1-z)(1-y)^2.
	\end{align}
\end{subequations}
The basic reproduction number as defined in Eq.\ \eqref{eq_basic_rep} is, once again, $R_0=\tau$. The disease-free equilibrium $(0,1)$ is a stable node for $R_0<1$, unstable node for $R_0>1$ and is undetermined otherwise. The unique endemic equilibrium has $y$-coordinate \cite{AchterbergMieghemAdaptiveMeanField}
\begin{equation*}
    \yinfty = 1-\frac{1-2\omega}{2\tau} - \sqrt{\bigg(\frac{1-2\omega}{2\tau}\bigg)^2 + \frac{2\omega}{\tau}},
\end{equation*}
and the EE becomes $\left(\yinfty, \frac{1}{\tau(1-\yinfty)}\right)$. Using basic arithmetic, it can be verified that $\tau>1$ implies $0<\yinfty\leq 1-\frac{1}{\tau}$, which ensures that the EE is contained in the physical region $(0,1)^2$. Thus, the EE exists for $R_0>1$.

The calculations needed for the stability of the EE become extremely cumbersome; however, the Bendixson-Dulac theorem, the uniqueness of the EE, the boundedness of solutions (see Lemma \ref{lem:bounded}) and the instability of the DFE ensure that the EE is globally asymptotically stable when $R_0>1$ (recall Remark \ref{rem:homocl}).

We summarize the stability of the two equilibria in the following table and present simulations of the two possible behaviours of system \eqref{ASIS_system} in Figure \ref{fig_ex3}.

\begin{table}[H]
	\centering
	\begin{tabular}{ccc}
		\textbf{Example 3:} $\fbr(y)=2y(1-y),\fcr(y)=(1-y)^2$ & if $R_0 \leq 1$ & if $R_0 > 1$ \\
		\hline
		Disease-free state $\left(0,1\right)$ & stable node & unstable node \\
		Endemic equilibrium $\left(\yinfty, \frac{1}{\tau(1-\yinfty)}\right)$ & unstable spiral & stable spiral
	\end{tabular}
\end{table}

\begin{figure}[H]
	\centering
	\subfloat[\label{fig_ex3_below} $R_0=0.8$]{%
		\includegraphics[width=0.49\textwidth]{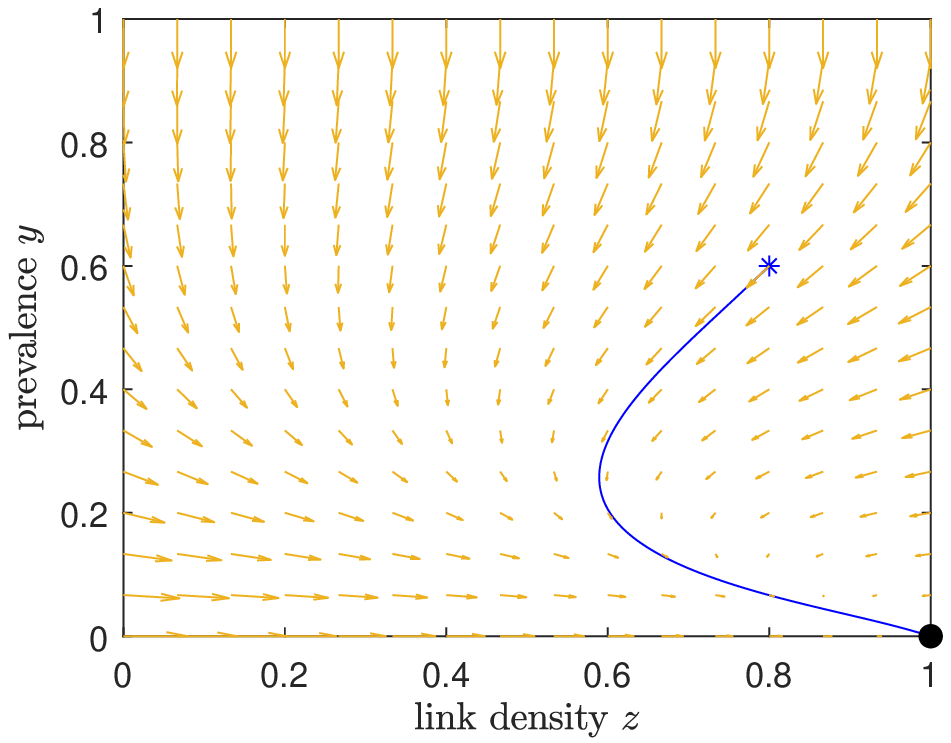}
	}
	\subfloat[\label{fig_ex3_above} $R_0=2.0$]{%
		\includegraphics[width=0.49\textwidth]{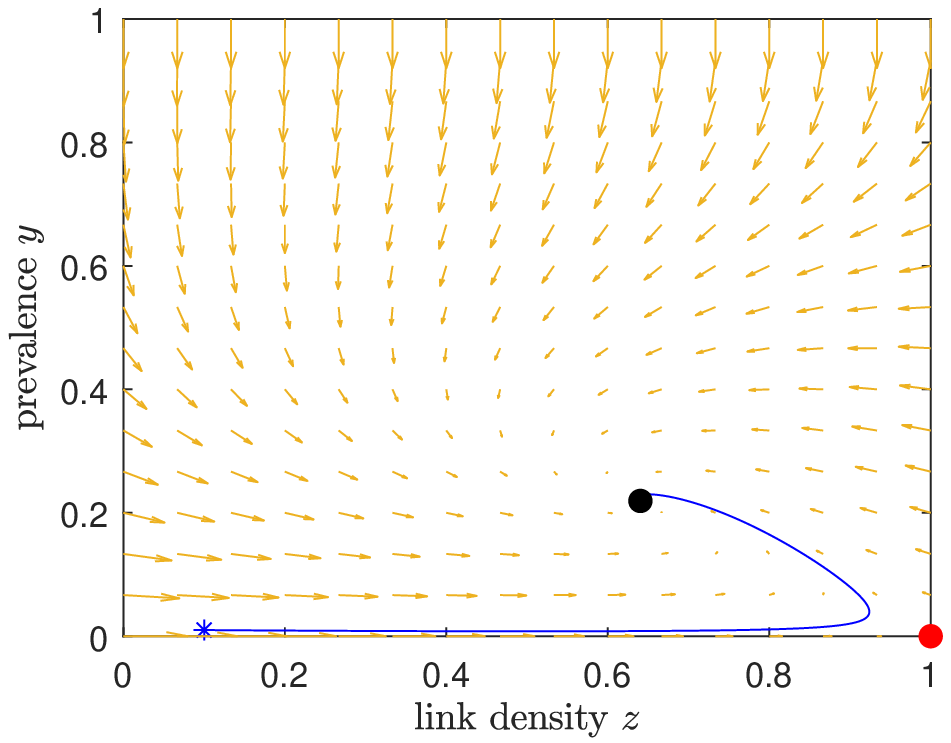}
	}
	\caption{Dynamics for Example 3. Starting point: asterisk; stable equilibrium: black dot; unstable equilibrium: red dot. (a) If $R_0<1$, any initial condition appears to converge to the DFE. (b) If $R_0>1$, the unique EE is globally stable. The other parameters are, for simplicity, $\zeta=\xi=1$.}
	\label{fig_ex3}
\end{figure}

\subsection{Example 4: Information spread}
In this section we consider an example from opinion dynamics, where a rumour is spreading in a population. The rumour is assumed to be attractive; hence, links are created to enhance the rumour spread. We use here the Adaptive Information Diffusion (AID) model, introduced by Trajanovski \emph{et al}.~\cite{Trajanovski2015} to describe the spread of information.

The prevalence can be interpreted as the fraction of the population that knows the rumour. Infection is equivalent to hearing the news and curing corresponds to forgetting the news. As a link-breaking response, we consider $\fbr(y) = (1-y)^2$, which reduces the link density when the prevalence is low. On the other hand, the link-creation response $\fcr(y)=2y(1-y)$ is based on the fact that the gossip is worth knowing, and thus the link density increases for larger prevalence $y$. However, when the news is only slightly present, little people may transmit the news to their neighbours, thereby we multiplied by the factor $(1-y)$. The factor 2 is again a conversion factor from \cite{AchterbergMieghemAdaptiveMeanField}. To summarise, the system of ODEs is given by:
\begin{subequations}\label{info_system}
	\begin{align}
		\df{y}{t} &= - y + \tau y (1-y) z, \\
		\df{z}{t} &= - \zeta z(1-y)^2 + 2\xi y (1-z)(1-y).
	\end{align}
\end{subequations}
The basic reproduction number $R_0$ cannot be determined in the traditional way using \eqref{eq_basic_rep}, as the disease-free equilibrium does not lose stability. Instead, we \emph{define} the basic reproduction number as the point where the two endemic equilibria are born (i.e.\ where \eqref{eq_ex4_EE} has non-complex solutions). Then the basic reproduction number follows as \cite{AchterbergMieghemAdaptiveMeanField}
\begin{equation}\label{eq_ex4_threshold}
	R_0 =\dfrac{2\tau}{\omega+2+\sqrt{8\omega}}.
\end{equation}
The disease-free equilibrium $(0,0)$ is stable for all $\tau>0$. The $y$-coordinates of the two endemic equilibria are given by \cite{AchterbergMieghemAdaptiveMeanField}
\begin{equation}\label{eq_ex4_EE}
	(\yinfty)_{1,2} = \frac{2\tau + \omega - 2 \pm \sqrt{(2\tau + \omega - 2)^2 - 8\tau \omega}}{4\tau}
\end{equation}
and the EE become $\left(\yinfty, \frac{1}{\tau(1-\yinfty)}\right)$.
The dynamics of the AID model is plotted in Figure~\ref{fig_ex4}. For $R_0 < 1$, the solution converges to $(0,0)$. For $R_0 > 1$, the solution may converge to the disease-free state $(0,0)$, but also to the endemic equilibrium, depending on the initial condition. The dependence of the basic reproduction number $R_0$ on the effective link-breaking rate $\omega$ is non-linear, which contrasts all earlier examples, that were either independent or linearly dependent on the effective link-breaking rate~$\omega$. Lastly, since the DFE is in the attracting part of the $z$-axis, we can not in general rule out the existence of a homoclinic orbit from the DFE.

We summarize the stability of the two equilibria in the following table and present simulations of the two possible behaviours of system \eqref{ASIS_system} in Figure \ref{fig_ex4}.

\begin{table}[H]
	\centering
	\begin{tabular}{cccc}
		\textbf{Example 4:} $\fbr(y)=(1-y)^2,\fcr(y)=2y(1-y)$ & if $R_0 < 1$ & if $R_0 \geq 1$ \\
		\hline
		Disease-free state $\left(0,0\right)$ & stable node & stable node \\
		Endemic equilibrium 1 $\left((\yinfty)_1, \frac{1}{\tau(1-(\yinfty)_1)}\right)$ & non-existent & unstable node \\
		Endemic equilibrium 2 $\left((\yinfty)_2, \frac{1}{\tau(1-(\yinfty)_2)}\right)$ & non-existent & stable node
	\end{tabular}
\end{table}

\begin{figure}[H]
	\centering
	\subfloat[\label{fig_ex4_below} $R_0=0.27$]{%
		\includegraphics[width=0.49\textwidth]{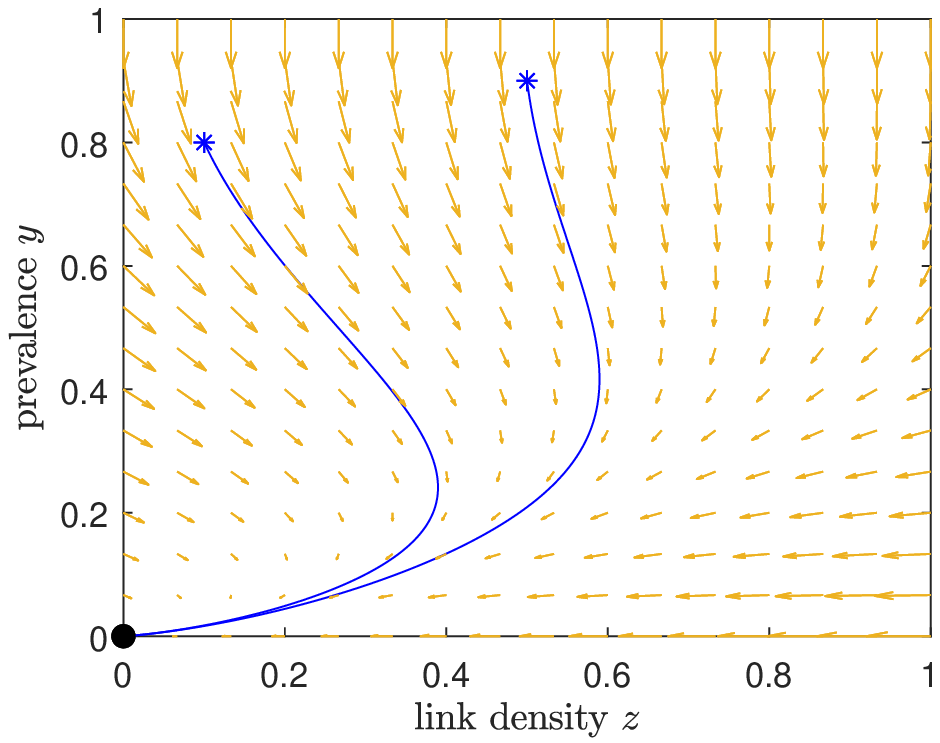}
	}
	\subfloat[\label{fig_ex4_above} $R_0=1.03$]{%
		\includegraphics[width=0.49\textwidth]{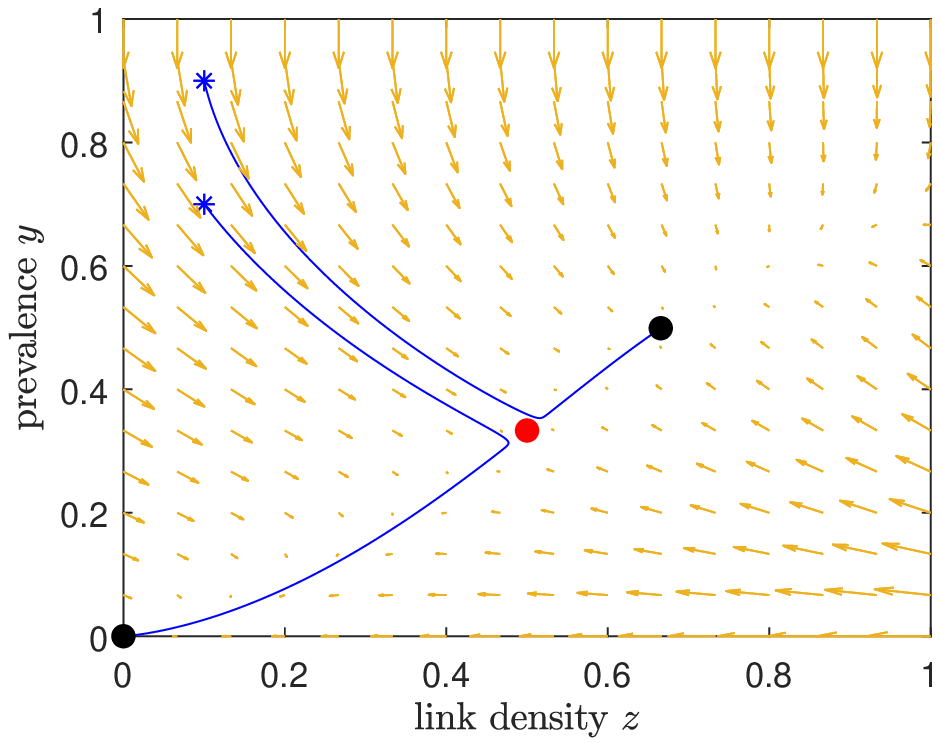}
	}
	\caption{Dynamics for Example 4. Starting points: asterisks; stable equilibrium: black dot; unstable equilibrium: red dot. (a) If $R_0<1$, any initial condition converges to the DFE. (b) If $R_0>1$, solutions may converge to the stable endemic equilibrium or the disease-free state, depending on the initial condition. The other parameters are, for simplicity, $\zeta=\xi=1$.}
	\label{fig_ex4}
\end{figure}

The basin of attraction of each stable equilibrium can be determined using a Lyapunov function. Such Lyapunov functions may distinguish for which initial conditions the system will converge to either the DFE or the stable EE. However, up to the best of the authors knowledge, no exact Lyapunov function can be constructed for system (\ref{eq_aNIMFA}) nor for most choices of the link-breaking and link-creation mechanisms.

However, the Lyapunov function can be approximated by considering a linearisation around a fixed point. For example, for the DFE $(0,0)$, its Jacobian equals
\begin{equation*}
    J_{(0,0)} = \begin{pmatrix}
        -1 & 0 \\
        2\xi & -\zeta
    \end{pmatrix}.
\end{equation*}
According to Khalil \cite[p.\ 73--80]{BookKhalil}, we can obtain an approximate Lyapunov function $\hat{V}$ by solving for the matrix $P$ in the following matrix equation
\begin{equation*}
    P J + J^T P = -I,
\end{equation*}
and the Lyapunov function follows as
\begin{equation*}
    \hat{V}(y,z) = \begin{pmatrix} y \\ z \end{pmatrix}^T P \begin{pmatrix} y \\ z \end{pmatrix}.
\end{equation*}
The estimated Region of Attraction $\Omega$ is then determined by the largest $c>0$ for which
$$
\Omega_c:=\{ (y,z) \in [0,1]^2 \, | \, \hat{V}(y,z) \leq c  \},
$$
is such that 
$$\Omega:=\max_{c >0} \left\{ \df{}{t} \hat{V}(\Omega_c) < 0 \right\}.$$
For Example 4, the estimated Lyapunov function around $(0,0)$ becomes
\begin{equation*}
    \hat{V}(y,z) = \frac{1}{2 \zeta} z^2 + \frac{1}{2} y^2 + \frac{2 \xi}{\zeta(1+\zeta)} yz + \frac{2 \xi^2}{\zeta(1+\zeta)} y^2
\end{equation*}
which is a tedious formula, but it is clear that $\hat{V} > 0$ in the biologically relevant region $[0,1]^2$. Unfortunately, the derivative $\df{}{t} \hat{V}$ is extremely complicated, even in such a simple case. Hence, we derive the largest possible approximate region of attraction numerically, which is shown in Figure~\ref{fig_ex4_ROA}. The approximate regions of attraction for the disease-free equilibrium $(0,0)$ and the stable endemic equilibrium are shown in orange, whereas the exact boundary separating the two regions, and thus the actual basins of attraction of the two equilibria, is shown as a light-blue curve. The estimated regions of attraction often poorly match with the true regions of attraction \cite{BookKhalil}, which is especially true for the stable EE in Figure~\ref{fig_ex4_ROA}. On the other hand, the region of attraction for the DFE is reasonably accurate.

\begin{figure}[H]
    \centering
    \includegraphics[width=0.4\textwidth]{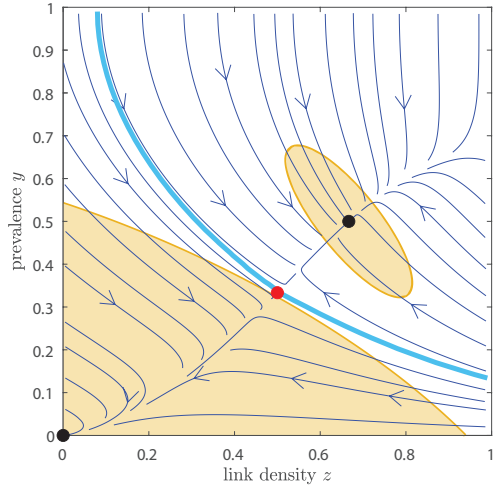}
    \caption{Regions of attraction for Example 4. The two attraction regions are separated by the numerically determined light-blue curve, which is the stable manifold of the unstable equilibrium. The approximate regions of attraction for the disease-free equilibrium and the stable endemic equilibrium are shown in orange. Black dots denote the stable equilibria and the red dot the unstable equilibrium. The parameters are, for simplicity, $\tau=3$, $\zeta=\xi=1$.}
    \label{fig_ex4_ROA}
\end{figure}

\section{Slow network dynamics}\label{sec_slow_network}
Suppose now that the network dynamics is slow compared to the disease spreading, that is to say, the disease is transmitted almost instantaneously when compared to the creation and removal of links in the network. Analytically, this translates in the introduction of a small parameter $0<\eps\ll 1$ in the system. We perform the substitutions $\zeta \mapsto \zeta\eps$ and $\xi \mapsto \eps$ such that
\begin{align}
	\begin{split}
	\df{y}{t} &= -y + \tau y (1-y) z, \\
	\df{z}{t} &= -\zeta \eps z \fbr(y) + \eps (1-z) \fcr(y).
	\end{split}\label{eq_slow_fast}
\end{align}
As before, the initial conditions are $y(0)\in [0,1]$ and $z(0)\in [0,1]$. We assume that all parameters (including initial conditions) are $\O(1)$-terms.

We intend to analyse Eq.\ (\ref{eq_slow_fast}) using Geometric Singular Perturbation Theory. System (\ref{eq_slow_fast}) is in standard GSPT form, and expressed in terms of the fast time variable $t$. In the limit $\eps\to 0$, we obtain the so-called \emph{layer equation}, or \emph{fast subsystem}:
\begin{align}\label{fast}
	\begin{split}
		\df{y}{t} &= -y + \tau y (1-y) z, \\
		\df{z}{t} &= 0.
	\end{split}
\end{align}
The corresponding critical manifold $\mathcal{C}_0$ is given by the union of the sets
\begin{equation}
    \mathcal{C}_0 =\{ (y,z)\in [0,1]^2 \ |\ y=0\} \cup \bigg\{ (y,z)\in [0,1]^2 \ \bigg|\ y= \frac{\tau z-1}{\tau z} \bigg\}.
\end{equation}
Notice that the second branch lies in the biologically relevant quadrant of $\mathbb{R}^2$ only for $z \geq 1/\tau$, and for this branch to have a non-empty intersection with $[0,1]^2$, we necessarily need $\tau >1$. 

Linearising the first equation of \eqref{fast} and evaluating it on $y=0$, we observe that the corresponding eigenvalue is 
$$
\lambda = \tau z -1,
$$
whereas the linearisation on the second branch of $\mathcal{C}_0$ gives
$$
\lambda=1 - \tau z.
$$
Clearly, the two branches of the critical manifold exchange stability at $(y,z)=(0,\frac{1}{\tau})$, which is a non-hyperbolic point. A visualisation of the stability of the two branches of $\mathcal{C}_0$ is shown in Figure \ref{fig_fast_slow}. 

\begin{figure}[H]
	\centering
	\begin{tikzpicture}
    	\node at (0,0){\includegraphics[width=0.4\textwidth]{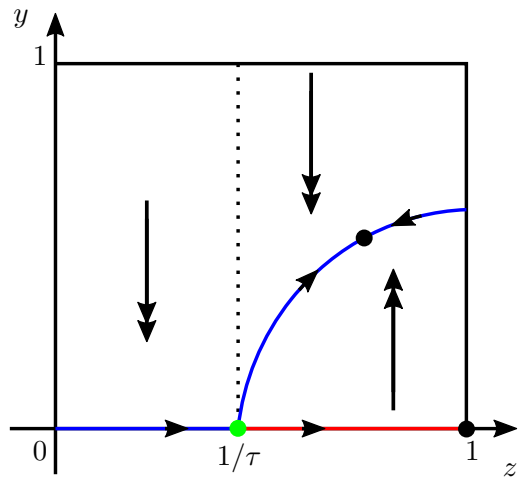}};
    	\node at (3.25,-3) {$z$};
    	\node at (-0.35,-2.85) {$1/\tau$};
    	\node at (-3.25,3) {$y$};
    	\node at (2.75,-2.75) {$1$};
    	\node at (-3,2.5) {$1$};
    	\node at (-3,-2.75) {$0$};
	\end{tikzpicture}
	\caption{Stability of the branches of the critical manifold $\mathcal{C}_0$ of  (\ref{eq_slow_fast}). Blue: stable; red: unstable. Green dot: non-hyperbolic point; black dots: equilibria. Double arrows: fast flow; single arrows: slow flow.}
	\label{fig_fast_slow}
\end{figure}

We now rescale time, introducing the slow time variable $s=\eps t$. System (\ref{eq_slow_fast}) becomes
\begin{align}
	\begin{split}
	\eps\df{y}{s} &= -y + \tau y (1-y) z, \\
	\df{z}{s} &= -\zeta  z \fbr(y) + (1-z) \fcr(y).
	\end{split}\label{slow}
\end{align}
Taking the limit as $\eps \rightarrow 0$, we obtain the so-called \emph{reduced subsystem}. The first equation defines once again the critical manifold $\mathcal{C}_0$: substituting $y$ in the second equation, we obtain one equation for the slow dynamics on the first branch of $\mathcal{C}_0$
\begin{equation}
    \df{z}{s} = -\zeta z \fbr(0) + \xi (1-z) \fcr(0),
    \label{slow_lim_gen_0}
\end{equation}
and one for the second
\begin{align}
	\df{z}{s} &= -\zeta z \fbr\left(\frac{\tau z-1}{\tau z}\right) + \xi (1-z) \fcr\left(\frac{\tau z-1}{\tau z}\right).\label{slow_lim_gen}
\end{align}
Without further specification for the functional responses $\fbr$ and $\fcr$, we can hardly deduce information on the asymptotic behaviour of the system. Hence, we return to Example 2 and consider $\fbr(y)=y$ and $\fcr(y)=1$ such that Eqs.\ \eqref{slow_lim_gen_0} and \eqref{slow_lim_gen} become, respectively, 
$$
\df{z}{s} = \xi (1-z),
$$
and
\begin{align}
	\df{z}{s} &= -\zeta \frac{\tau z-1}{\tau } + \xi (1-z).
\end{align}
The corresponding steady-states are $z_0=1$, representing the DFE, which the system tends towards if $y(0)=0$, and 
\begin{equation}
	\zinfty = \frac{\omega+\tau}{\tau \omega+\tau}\in \bigg(\frac{1}{\tau},1\bigg), \qquad \text{ if } \tau>1,
\end{equation}
representing the EE, which is globally asymptotically stable for orbits on the second branch of $\mathcal{C}_0$. Given that $\tau > 1$ and further assuming that $y(0) > 0$, two difference kinds of behaviour exist.

If $z(0)<1/\tau$, the system quickly approaches the line $y=0$, where $z$ starts to increase. We observe a delayed loss of stability, and we can approximate the dynamics in a neighbourhood of $y=0$ with the so-called entry-exit function \cite{de2016entry,neishtadt1987persistence,neishtadt1988persistence}. An orbit entering a neighbourhood of $y=0$ at a point with $z$-coordinate $z=z_{\text{in}}$ will exit the same neighbourhood at a point with $z$-coordinate $z=z_{\text{out}} >1/\tau> z_{\text{in}}$. The value $z_{\text{out}}$ is given implicitly as the unique solution of
$$
\int_{z_{\text{in}}}^{z_{\text{out}}} \dfrac{\tau z - 1}{1-z} \text{d}z =0.
$$
Since the integrand function diverges at $+\infty$ as $z\rightarrow 1$, it follows that the exit point $z_{\text{out}}$ will be strictly smaller than $1$ for any $z(0) \in [0,1/\tau)$. We refer to \cite[Sec. 3]{jardon2021geometric} for a detailed analysis of a similar entry-exit function, derived from a different epidemiological model.

If $z(0)>1/\tau$, the fast dynamics brings the system close to the second branch of $\mathcal{C}_0$; once an orbit reaches an $\mathcal{O}(\eps)$ neighbourhood of this curve, the slow dynamics tends asymptotically towards $\zinfty$.

To summarize, if $y(0)=0$, the system tends towards the equilibrium $(y,z)=(0,1)$. If instead $y(0)\in(0,1]$, the system converges, possibly after a slow passage near $y=0$ which represents a ``dormant'' phase for the infection, towards the endemic equilibrium
\begin{equation*}
    (\yinfty,\zinfty) = \left( \frac{\tau-1}{\omega+\tau}, \frac{\omega+\tau}{\tau\omega+\tau} \right) \in (0,1]^2.
\end{equation*}
We emphasise that this holds true only for $R_0=\tau>1$. Figure~\ref{fig_ex_fast_slow} provides a numerical simulation that shows the entry-exit phenomenon. 

\begin{figure}[H]
    \centering
    \includegraphics[width=0.6\textwidth]{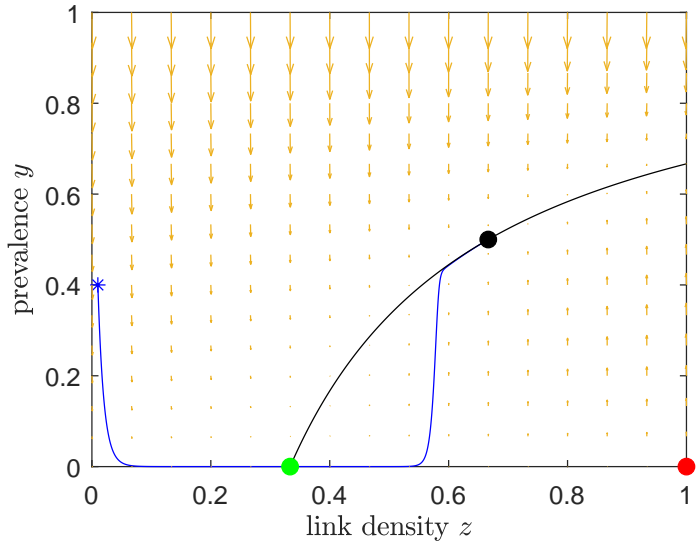}
    \caption{The entry-exit behaviour of the fast-flow system (\ref{eq_slow_fast}) for Example 2. Parameters are $R_0=\tau=3, \xi=\zeta=1, \eps=0.01$. The blue trajectory starts at the asterisk, exhibits a slow passage close to $y=0$ and then converges to the stable EE, indicated by the black dot. The red dot is the unstable equilibrium and the green dot is the non-hyperbolic fixed point. The solid black line represents the branch of the critical manifold characterized by $y>0$.}
    \label{fig_ex_fast_slow}
\end{figure}

\section{Conclusion}\label{sec_conclusion}
In this paper, we developed a minimal model for modelling an SIS disease spread with personal contact avoidance, called adaptive NIMFA (aNIMFA). We investigated local and global stability of the model and showed that limit cycles cannot exist. Furthermore, we analysed various examples in detail, from epidemic contagion to information spread.

In this work, we assumed an homogeneous mixing of the population. In reality, this homogeneity is often unrealistic; some people have frequent contacts while other people never meet. We expect that one can extend the current results for a community of subpopulations, on a network with $N$ nodes, as it was done in \cite{ottaviano2022globalpr} in order to generalize the results obtained in \cite{ottaviano2022global} for SAIRS compartmental models. While considering subpopulations, one must decide whether the link-breaking and link-creation functional responses act on the local prevalence of the node or on the global prevalence of the whole network. From a modelling perspective, we see possibilities for both approaches, or a mix of these \cite{zhang2022interplayawarenessdisease}.

We see several other interesting directions for future research. For example, is it possible to provide, besides continuity, conditions on $\fbr$ and $\fcr$ such that we can limit/bound the number of endemic equilibria from Eq.\ (\ref{eq_EE_implicit})? Can we determine for which $\fbr$ and $\fcr$ the endemic equilibrium is unique?

Moreover, for other types of infectious diseases, it could be beneficial to consider the opposite slow-fast decomposition, compared to the one we analysed in Section \ref{sec_slow_network}. Namely, one could consider the network to be much faster than the spread of the disease, possibly including an Exposed or Asymptomatic compartment through which Susceptible individuals need to pass before becoming Infected and infectious. As a final comment, we mention the possibility to include delays into the knowledge about the current prevalence. As the COVID-19 pandemic exemplified, testing an individual typically takes several hours or days before the result is communicated. Moreover, the daily reported cases by governmental agencies typically run a few days behind. One modelling approach is to convert the aNIMFA model into a delay-differential equation, which typically complicates the analysis significantly. We leave these possibilities as an outlook for future works.

{\footnotesize
	\bibliographystyle{unsrt}
	\bibliography{references}
}

\end{document}